\documentclass[11pt]{amsart}
\usepackage{epsfig}
\usepackage[usenames,dvipsnames]{color}
\usepackage[colorlinks=true,linkcolor=blue,citecolor=BrickRed]{hyperref}

\newtheorem{thm}{Theorem}[section]
\newtheorem{lem}[thm]{Lemma}

\newtheorem{prop}[thm]{Proposition}

\theoremstyle{definition}

\newtheorem{note}[thm]{Note}

\theoremstyle{remark}

\newcommand{\R}{\mathbf{R}}
\newcommand{\RP}{\mathbf{RP}}
\newcommand{\Z}{\mathbf{Z}}

\newcommand{\ol}[1]{{\overline #1}}

\renewcommand{\tilde}{\widetilde}
\newcommand{\C}{\mathcal{C}}

\renewcommand{\S}{\mathbf{S}}
\renewcommand{\H}{\mathbf{H}}
\renewcommand{\Z}{\mathbf{Z}}

\begin{document}


\title[Vertices of closed curves in  Riemannian surfaces]{Vertices of closed curves in Riemannian surfaces}

\author{Mohammad Ghomi}
\address{School of Mathematics, Georgia Institute of Technology,
Atlanta, GA 30332}
\email{ghomi@math.gatech.edu}
\urladdr{www.math.gatech.edu/$\sim$ghomi}
\date{Last Typeset \today.}
\thanks{Supported by NSF Grant DMS-0336455, and CAREER award DMS-0332333.}

\begin{abstract}
We uncover some connections between the topology of a complete Riemannian surface $M$ and the minimum number of \emph{vertices}, i.e., critical points of geodesic curvature, of closed curves in $M$. In particular
we show that the space forms with finite fundamental group are the only  surfaces  in which every simple closed curve has  more than two vertices. Further we characterize the  simply connected  space forms as the only  surfaces in which every closed curve bounding a compact immersed surface has more than two vertices.
\end{abstract}

\maketitle

\section{Introduction}
In this paper we study the relation between the topology of a complete Riemannian surface $M$ and the minimum number of \emph{vertices}, i.e., critical points of geodesic curvature, of closed curves in $M$. Our prime motivation here
is   the  classical theorem of 
 Kneser \cite{kneser, montiel&ros}, which states that
 any $(\C^2)$ simple closed curve in Euclidean plane $\R^2$  has at least four vertices. It is known that this result also holds in the sphere $\S^2$ and hyperbolic plane $\H^2$, since the stereographic projection and the inclusion map of the Poincar\'{e} disk preserve vertices  \cite{maeda,costa&firer}. Further, it follows from another  classical  result due to M\"{o}bius \cite{mobius,Ssasaki} that any simple closed curve in the projective plane $\RP^2$ has at least three vertices. We show that, up to a rescaling, these are the only  surfaces where every simple closed curve has more than two vertices: 
 
 \begin{thm}\label{thm:1}
The only complete Riemannian surfaces where every simple closed curve has more than two vertices are the space forms with finite fundamental group.
\end{thm}

Thus the simply connected space forms are the only complete Riemannian surfaces where Kneser's four vertex theorem holds, and real projective planes are the only other surfaces where every simple closed curve has at least three vertices.
Another classification result we obtain  in this work uses the extension of Kneser's theorem to self-intersecting curves by Pinkall  \cite{pinkall} who showed that every closed curve in $\R^2$ (and hence $\H^2$ and $\S^2$) which bounds a compact immersed region must have at least four vertices. Our next theorem  shows that this property also characterizes the simply connected space forms:

\begin{thm}\label{thm:2}
The only complete Riemannian surfaces where every  closed curve which bounds a compact immersed surface has more than two vertices are the simply connected space forms.   
\end{thm}
 
 Finally one may ask what are all surfaces which satisfy the condition of the last theorem if the word ``immersed" is replaced by ``embedded". Obviously, surfaces of  the above theorems would then fall into that category, but there are  more examples:

\begin{thm}\label{thm:3}
The only complete Riemannian surfaces where every  closed curve which bounds a compact \emph{embedded surface} has more than two vertices are  orientable space forms of genus zero, flat tori, and rescalings of $\RP^2$. 
\end{thm}

We recall that an orientable surface of genus zero is one that is homeomorphic to $\S^2$ minus a totally disconnected subset \cite{richards}.  The above theorems are proved in the next three sections below. The proofs of Theorems \ref{thm:1} and \ref{thm:3} employ some basic curve shortening and certain perturbation results for closed geodesics or horocycles in hyperbolic surfaces. An important lemma utilized here is a result of Jackson \cite{jackson:vertices} who showed that Kneser's four vertex property holds only in surfaces with constant curvature. The proof of Theorem \ref{thm:2} also uses this lemma, together with an explicit example of a rather remarkable cylindrical curve which has only two vertices but bounds a compact immersed surface (Figure \ref{fig:cylinder}). A general result for perturbing closed geodesics to curves with only two vertices is discussed in the appendix.

Four vertex theorems have spawned a vast and diverse literature since the first version of the theorem was
 proved in 1909 by Mukhopadhyaya \cite{mukhopadhyaya}, who showed that convex planar curves have four vertices.  This is the only version mentioned in nearly all differential geometry textbooks, with the exceptions of \cite{guggenheimer, montiel&ros} where the more general theorem of Kneser is discussed, see also \cite{jackson:bulletin, osserman:vertex, heil}.  Pinkall's paper \cite{pinkall} and related work \cite{umehara,cairns} offer other more general proofs as well, based on a number of ideas. For more recent proofs see \cite{angenent,sussman} which use curvature flow and Sturm theory. Another recent development is concerned with a converse of Kneser's theorem \cite{dahlberg, gluck:vertex}. There are also interesting generalizations to space curves \cite{sedykh:vertex, umehara2, thorbergsson&umehara},  connections with contact geometry \cite{arnold:flattening, uribe-vargas}, polygonal analogues \cite{pak:book}, and a version for surfaces with boundary \cite{ghomi:verticesA}. See \cite{rmbc} for a physical application, and \cite{gluck:notices, ovsienko&tabachnikov} for some more references and historical remarks.

\section{Proof  of Theorem \ref{thm:1}}\label{sec:thm1}

We begin by recording the result of Jackson mentioned above:

\begin{lem}[\cite{jackson:vertices}]\label{lem:jackson}
Let $M$ be a Riemannian surface with curvature $K$ and let $p$ be a point of $M$. Suppose that $p$ is not a stationary point of $K$, i.e, $dK_p\neq 0$. Then there exists an $\epsilon>0$ such that for all $0<r\leq \epsilon$, the metric circle of radius $r$ centered at $p$ has only two vertices.\qed
\end{lem}

Thus we may confine our attention to complete Riemannian surfaces $M$ of constant curvature $K$, or $2$-dimensional \emph{space forms},  as far as proving Theorem \ref{thm:1} is concerned. Furthermore, we may assume that $M$ is not simply connected due to the following fact: let $\S^2\subset \R^3$ denote the unit sphere and $\H^2\subset\R^2$ be the Poincar\'{e}  half-plane with its standard metric of constant curvature $-1$; then we have:

\begin{lem}[\cite{maeda}]\label{lem:maeda}
The stereographic projection $\pi\colon \S^2-\{(0,0,1)\}\to\R^2$ and the inclusion map $i\colon\H^2\to\R^2$ preserve the sign of the derivative of the geodesic curvature of curves.\qed
\end{lem}

Thus Kneser's theorem holds in all simply connected $2$-dimensional  space forms.
The property of the stereographic projection mentioned above was already known to Kneser \cite{kneser}, and  the property of the inclusion map for the Poincar\'{e} disk was proved by Maeda \cite{maeda}; but the Poincar\'{e} disk and half-plane are equivalent up to a M\"{o}bius transformation, which establishes the corresponding fact for the Poincar\'{e} half-plane, since M\"{o}bius transformations  preserve vertices \cite{pinkall,heil}. Essentially, these observations are consequences of the fact that the stereographic projection and the inclusion map of the upper half-plane send circles to circles, and the vertices occur when the curve has third order contact with its osculating circle.

It remains  to consider non-simply connected $2$-dimensional space forms $M$. We may assume, after a rescaling, that the curvature $K$ of $M$ is $0$, $1$, or $-1$. Then by the  Hopf-Killing  theorem \cite{stillwell,ratcliffe}, $M=X/G$ where  $X=\R^2$, $\S^2$, or $\H^2$ and $G$ is a discreet subgroup of isometries of $X$ which acts freely and properly discontinuously on $X$. In particular, the projection $\pi\colon X\to X/G=M$ is a Riemannian covering map. We also recall that $G$ is isometric to the fundamental group $\pi_1(M)$. Indeed, for any point $o\in M$, $\ol o\in X$ with $\pi(\ol o)=o$, and  closed curve $\gamma\colon [0,L]\to M$ with $\gamma(0)=o$, if  $\ol\gamma\colon [0,L]\to X$ is the lifting of $\gamma$ with $\ol\gamma(0)=\ol o$, then there exists a unique element $g\in G$ such $g(\ol\gamma(0))=\ol\gamma(L)$. This correspondence, which will be utilized a number of times below, establishes the isomorphism between $G$ and $\pi_1(M)$.

\subsection{The elliptic case}\label{subsec:elliptic}
If $K=1$, then $M$ is  the real projective plane $\RP^2=\S^2/\{\pm 1\}$. By a theorem of M\"{o}bius \cite{mobius,Ssasaki}, see Note \ref{note:mobius} below for other references, any simple closed noncontractible curve $\Gamma$ in $\RP^2$ has at least three \emph{inflection points}, i.e., points where the geodesic curvature vanishes. But there must be at least one vertex between every pair of inflections. Thus $\Gamma$ must have at least three vertices. On the other hand, if $\Gamma$ is  contractible, then it lifts to a pair of simple closed curves in $\S^2$, each of which must have four vertices by Kneser's theorem. Since the covering map is one-to-one on each of these curves, it then follows that $\Gamma$ must have at least four vertices. So we conclude then that all simple closed curves in $\RP^2$ must have at least three vertices. (Note that in a nonorientable surface, the geodesic curvature in only well-defined locally and up to a sign; however, this is enough to allow one to talk about vertices and inflection points.) 

In the remaining cases we construct explicit examples of simple closed curves with only two vertices, by perturbing geodesics or horocycles of $M$.

\subsection{The parabolic case}\label{subsec:parabolic}
If $K=0$, then $M=\R^2/G$ where there are exactly four types of possibilities for $G$ corresponding to the cases of  cylinder, twisted cylinder, torus, or Klein bottle \cite{stillwell}. In each of these cases $M$ contains a simple closed geodesic $\Gamma$, as may be easily seen by looking at the fundamental region of these surfaces which is either a rectangle or an infinite strip in $\R^2$. Let $\gamma\colon\R/L\to M$ be a parametrization of $\Gamma$ by arclength (where $L$ denotes the length of $\Gamma$), and $\ol\gamma$ be a lifting of $\gamma$ to $\R^2$. Then $\ol\gamma$ traces a line which we may assume to be the $x$-axis. Let $g\in G$ be the unique element such that $g(\ol\gamma(0))=\ol\gamma(L)$.  Then $g$ is either  the ``translation" $(x,y)\mapsto (x+L,y)$ or the ``glide reflection" $(x,y)\mapsto (x+L,-y)$ and we define, respectively,  
$$
\tilde\gamma(t):=\left(t, \lambda \sin\left(\frac{2\pi t}{L}\right)\right) \quad\text{or} \quad 
\tilde\gamma(t):=\left(t, \lambda \sin\left(\frac{\pi t}{L}\right)\right),
$$
 for some $\lambda>0$. Then the image of $\tilde\gamma$ is invariant under the action of $g$. Consequently, if $\pi\colon \R^2\to M$ denotes the covering map, then $\pi\circ\tilde\gamma$ traces a smooth closed curve in $M$ with only one or two vertices. Further note that, as $\lambda\to 0$, $\tilde\gamma$ converges to $\ol\gamma$, and consequently $\pi\circ\tilde\gamma$ converges to $\pi\circ\ol\gamma=\gamma$ with respect to the $\C^1$-norm. So since $\gamma$ is simple, it follows that $\pi\circ\tilde\gamma$ will be simple as well for sufficiently small $\lambda$.

\subsection{The hyperbolic case}\label{sec:hyperbolic}
If $K=-1$, or $M$ is a hyperbolic surface, we need to establish a basic structure theorem first.  
 Recall that an \emph{end} of a surface is a nested sequence of subsets  which eventually lie outside any given compact subset. Each element of this sequence is called an \emph{end representative}, and two ends are equivalent if each representative of one end lies in a representative of the other.
A \emph{cusp} of a hyperbolic surface is an end with a representative  which is isometric to a representative $C$ of the ``thin" end of the \emph{parabolic cylinder} $\H^2/G$ where $G$ is generated by the parabolic translations $(x,y)\mapsto (x+L,y)$; more explicitly, a cusp representative  may be defined as
$$
C=C(h):=\big\{(x,y)\in \H^2\mid y>h\big\}/G,
$$
 for some $h>0$. Alternatively, a cusp may be visualized as an end of the tapering surface of revolution in $\R^3$ known as the \emph{pseudosphere}. 

\begin{prop}\label{prop:simpleclosed}
Every non-simply connected complete hyperbolic surface  contains a simple closed geodesic or a cusp.
\end{prop}

In the special case where the surface is orientable and has finite Euler characteristic, the above proposition follows from a result in the recent book of Borthwick \cite[Prop. 2.16]{borthwick}. With the aid of this result, and some curve shortening, we give the more general proof which we seek.
First we need to establish the following lemma.  If $M$ were orientable, this lemma  would  follow from basic results on curve shortening flow by curvature \cite{grayson, chou&zhu}; however, when $M$ is not orientable, the geodesic curvature is not well-defined along the entire curve; hence we use the ``disk flow" method devised in \cite{hass&scott} which does not require orientability. 

\begin{lem}\label{lem:isotope}
Let $M$ be a complete Riemannian surface  and $\Gamma$ be a simple closed   curve of length $L$ in $M$. Suppose that $\Gamma$ is not isotopic to a point and the set of all curves in the isotopy class of $\Gamma$ whose lengths are bounded above by $L$ are confined to a compact region of $M$. Then $\Gamma$ is isotopic to a simple closed geodesic.
\end{lem}
\begin{proof}
If $M$ is compact, the lemma follows immediately from  \cite[Thm. 1.8]{hass&scott}. But the proof of \cite[Thm. 1.8]{hass&scott} shows that compactness is needed only so that one can apply  Ascoli's theorem to conclude that any length non-increasing homotopy $\Gamma_t$  has a convergent subsequence. Of course one may draw the same conclusion as long as $\Gamma_t$ is confined to a compact subset of $M$, which we assume is the case.
\end{proof}

Next we use the previous lemma to establish another  basic fact:

\begin{lem}\label{lem:mobius}
Every complete  non-orientable Riemannian surface contains a simple closed geodesic, which has a tubular neighborhood homeomorphic to a M\"{o}bius strip.
\end{lem}
\begin{proof}
If a complete Riemannian surface $M$ is nonorientable then it must contain a M\"{o}bius band $U$. Let $\Gamma$ be a smooth simple closed curve which is a retraction of $U$. By the isotopy extension lemma \cite{hirsch:book}, if $\Gamma'$ is any curve in $M$ which is isotopic to $\Gamma$, then $\Gamma'$ will also have a neighborhood $U'$ which is homeomorphic to a M\"{o}bius band and retracts onto $\Gamma'$. This yields the following two observations. First, $\Gamma$ is not isotopic to a point, because  every surface is locally orientable. Second,  no curve homotopic to $\Gamma$ may be disjoint from $\Gamma$. This is due to the fact that there exists a smooth curve $\Gamma'$ homotopic to $\Gamma$ which intersects $\Gamma$ only once, see Figure \ref{fig:mobius}. 
\begin{figure}[h] 
   \centering
   \includegraphics[width=2in]{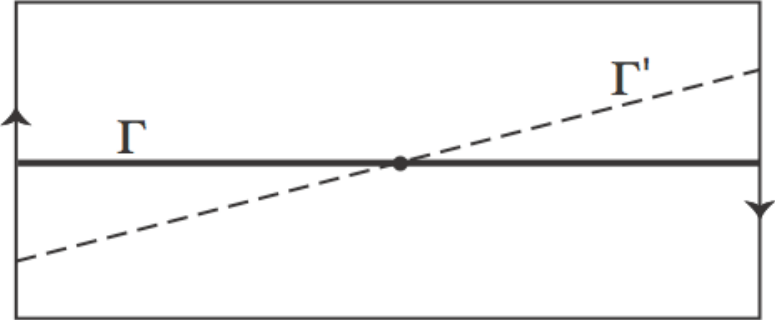} 
   \caption{}\label{fig:mobius}
\end{figure}
Next note that any curve $\Gamma''$ homotopic to $\Gamma$ will be homotopic to $\Gamma'$ as well. Suppose now, towards a contradiction, that $\Gamma''$ is disjoint from $\Gamma$. Then, if  $\#$ denotes the intersection number,  we have
$$0=\#(\Gamma'',\Gamma)= \#(\Gamma',\Gamma)=1,$$
by the invariance of intersection number under homotopy \cite{hirsch:book}.
Thus, $\Gamma''\cap\Gamma\neq\emptyset$. This shows that any \emph{length nonincreasing} isotopy of $\Gamma$ must be confined within a compact region of $M$ (if length of $\Gamma$ is $L$, then this region would consist of  points of $M$ which are within a distance $L/2$ of $\Gamma$).
It follows then, via Lemma \ref{lem:isotope}, that $\Gamma$ is isotopic to a simple closed geodesic.
\end{proof}

We need only one other basic result before proving Proposition \ref{prop:simpleclosed}. A \emph{hyperbolic cylinder} is the quotient $\H^2/G$ where $G$ is generated by the hyperbolic translations $(x,y)\mapsto (e^Lx, e^L y)$. Thus a hyperbolic cylinder contains a simple closed geodesic, i.e., the image of the positive half of the $y$-axis, which cuts the cylinder in half. Following Borthwick \cite{borthwick} we define a \emph{funnel}, as one  of these halves. We say that a surface is \emph{topologically finite} if it is homeomorphic to a compact surface minus a finite set of points.

\begin{lem}[\cite{borthwick}]\label{lem:borthwick}
Every non-simply connected noncompact  complete orientable hyperbolic surface with finite topology has a cusp or funnel end.
\end{lem}
\begin{proof}
If our hyperbolic surface $M=\H^2/G$ is homeomorphic to an annulus, or, equivalently, $M$ is orientable and $G$ has only one generator, then that generator is conjugate to either a parabolic or hyperbolic translation of $\H^2$, which we defined above. Consequently $M$ is isometric to either the parabolic cylinder or hyperbolic cylinder, in which case $M$ has a cusp or funnel end respectively, and we are done. On the other hand, if $M$ is not homeomorphic to an annulus (and is not simply connected), then $M$ will be ``nonelementary" and the proof follows from \cite[Thm. 2.13]{borthwick} which classifies the ends of nonelementary geometrically finite orientable hyperbolic surfaces. To apply this classification result, we just need to note that since $M$ is topologically finite, it is ``geometrically finite" as well \cite[Thm. 2.10]{borthwick}, i.e.,  $G$ is finitely generated, or, equivalently, the fundamental region of $M$ is a finite-sided convex polygon in $\H^2$.
\end{proof}

Now we are  ready to prove Proposition \ref{prop:simpleclosed}:

\begin{proof}[Proof of Proposition \ref{prop:simpleclosed}]
By Lemma \ref{lem:mobius} we may suppose that our surface, say $M$, is orientable.
It is well-known that every compact non-simply connected orientable surface contains a simple closed geodesic, e.g., this follows from Lemma \ref{lem:isotope}; or see \cite[Sec. 9.6]{ratcliffe} for the hyperbolic case. So we may also suppose that  $M$ is noncompact. 
Then, if $M$ has finite topology, by Lemma \ref{lem:borthwick}, one of the ends of $M$ must be a ``cusp" or a ``funnel" \cite{borthwick}. In either case we would be done since, by definition, funnels  are bounded by  simple closed geodesics. So it remains to consider the case where $M$ has infinite topology, although all we need is that $M$ have at least four ends. Then we show that $M$ must have a simple closed geodesic as follows.

By the generalized Jordan curve theorem, any simple closed curve $\Gamma_0$ divides $M$ into two components. Choose $\Gamma_0$ so that each of the components of $M-\Gamma_0$ contains at least two ends of $M$, and $\Gamma_0$ is rectifiable. Then let $\Gamma_t$ be a length  nonincreasing isotopy of  $\Gamma_0$, e.g., as defined by Hass and Scott \cite[Thm. 1.8]{hass&scott}. Note that, by the isotopy extension lemma \cite{hirsch:book}, the topology of $M-\Gamma_t$ does not depend on $t$. Let $A_t$ be a continuous choice of a component of $M-\Gamma_t$, and $E\subset A_0$ be an end  representative of $M$  which does not contain all ends of $M$ that are in  $A_0$. Then $\Gamma_t$ cannot be contained entirely in $E$ for any time $t$,  because then $A_t$ would be disjoint from some of the ends of $A_0$.  Now let $E'\subset E$ be another end representative such that the distance of $\partial E'$ from $\partial E$ is bigger than $L/2$ were $L$ is the length of $\Gamma_0$. Then $\Gamma_t\cap E'=\emptyset$ for all $t$.

Similarly, for each end of $M$  we may choose an end representative  which will always be disjoint from  $\Gamma_t$. Then the complement of all these end representatives is a  compact subset of $M$ which contains $\Gamma_t$ for all $t$, and so we may apply Lemma \ref{lem:isotope} to complete the proof.
\end{proof}

Having proved Proposition \ref{prop:simpleclosed}, we may now proceed with the rest of the proof of Theorem \ref{thm:1} by considering the following two cases:

\subsubsection{} If our hyperbolic surface $M=\H^2/G$ contains a cusp, then it contains a \emph{horocycle}, i.e., a simple  closed curve $\gamma\colon \R/L\to M$ which  lifts to a horizontal line $\ol\gamma(t)=(t,h)\subset\H^2$. Then let $\tilde\gamma$ be the perturbation of $\ol\gamma$ given by 
$$\tilde\gamma(t):=\left(t,\,h+\lambda \sin\left(\frac{2\pi t}{L}\right)\right),$$
 see Figure \ref{fig:horocycle}. 
\begin{figure}[h] 
   \centering
   \includegraphics[width=3.25in]{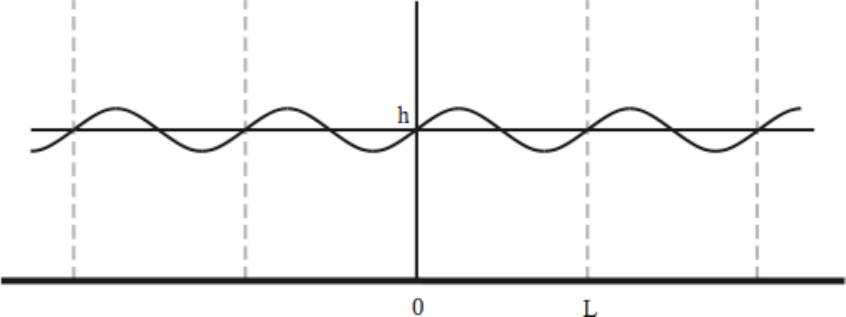} 
   \caption{}
   \label{fig:horocycle}
\end{figure}
Recall that there is  a unique element $g\in G$, given by $g(\ol\gamma(0))=\ol\gamma(L)$, which in this case is the parabolic translation $(x,y)\mapsto (x+L, y)$.
Further, as we argued in Section \ref{subsec:parabolic}, since $\tilde\gamma$ is also invariant with respect to $g$, the  projection $\pi\circ\tilde\gamma$  is a smooth closed curve in $M$ with only two vertices for any $\lambda\neq 0$, by Lemma \ref{lem:maeda}. Furthermore, again as $\lambda\to 0$, $\tilde\gamma$ converges to $\ol\gamma$ with respect to the $\C^1$ norm. Thus, since $\gamma$ is simple,  $\pi\circ\tilde\gamma$ will be simple as well.

\subsubsection{} If $M=\H^2/G$ contains a simple closed geodesic $\gamma\colon\R/L\to M$, then it lifts to a geodesic $\ol\gamma\colon\R \to \H^2$. We may assume that $\ol\gamma$ traces the upper half of the $y$-axis in its positive direction, and $\ol\gamma(0)=(0,1)$. Then $\ol\gamma(L)=(0,e^L)$ (recall that the hyperbolic distance of $(0,1)$ from $(0,y)$ is given by $\ln(y)$).
As before, let $g\in G$ be the (unique) element such that $g(\gamma(0))=\gamma(L)$. Then $g$ is either the hyperbolic translation $(x,y)\mapsto (e^L x, e^L y)$ or the glide reflection $(x,y)\mapsto (-e^L x, e^L y)$, depending on wether or not small tubular neighborhoods of $\gamma$ are orientable. So there are two cases to consider:

If $g$ is the hyperbolic translation (i.e., a tubular neighborhood of $\gamma$ in $M$ is orientable), let 
$$
\tilde\gamma(t):=\left(\lambda t \sin\left(\frac{2\pi}{L}\ln(t)\right),t\right),
$$
see Figure \ref{fig:esine}.
\begin{figure}[h] 
   \centering
   \includegraphics[width=2.75in]{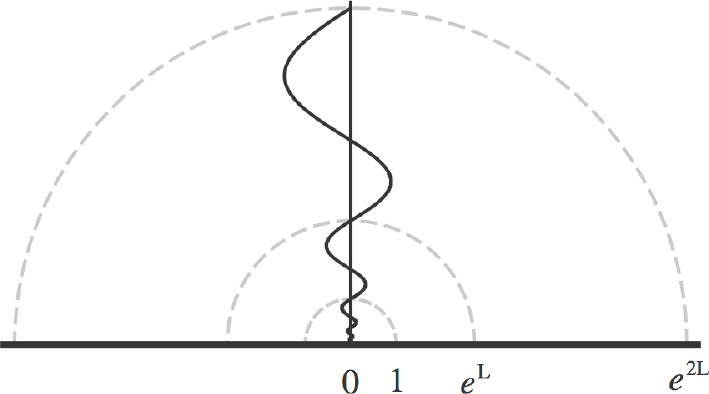} 
   \caption{}
   \label{fig:esine}
\end{figure}
Note that $\tilde\gamma(e^Lt)=e^ L \tilde\gamma(t)$, i.e., the image of $\tilde\gamma$ is invariant under the action of $g$.  Thus $\pi\circ\tilde\gamma$ is a smooth closed curve in $M$, and choosing $\lambda$ sufficiently small we can make sure that  $\pi\circ\tilde\gamma$ is simple as we argued in Section \ref{subsec:parabolic}. Finally note that $\tilde\gamma$ has only two vertices on the interval $[1,e^L]$. Thus,  the projection of $\pi\circ\tilde\gamma$ will have only two vertices, again by Lemma \ref{lem:maeda}.

If $g$ is the glide reflection (i.e., no neighborhood of $\gamma$ in $M$ is orientable)  let 
$$
\tilde\gamma(t):=\left(\lambda t \sin\left(\frac{\pi}{L}\ln(t)\right),t\right).
$$
Then the image of $\tilde\gamma$ is invariant under the action of $g$ and $\pi\circ\tilde\gamma$ again yields the desired curve for small $\lambda$, which completes the proof of Theorem \ref{thm:1}. 

We  record below the last part of the proof for future reference:

\begin{lem}\label{lem:approx}
Any closed geodesic or horocycle  in a hyperbolic surface is $\C^\infty$-close to a closed curve  with no more than two vertices. \qed
\end{lem}

\begin{note}\label{note:mobius}
The theorem of M\"{o}bius mentioned above is equivalent to the statement that every simple closed curve $\Gamma$  in $\S^2$ which is symmetric with respect to the antipodal reflection ($\Gamma=-\Gamma$) has at least six inflection points. Since the antipodal reflection switches the sign of geodesic curvature, the number of inflection points of $\Gamma$ must be $2m$ where $m$ is odd. So, it is enough to show that $\Gamma$ has more than two inflections. In this sense, the theorem of M\"{o}bius may be viewed as a special case of the ``tennis ball theorem" \cite{arnold:flattening, arnold:plane,angenent} which proves the existence of at least four inflections for curves which bisect the area of $\S^2$. The latter result in turn follows from a theorem of Segre \cite{segre:inflection} who proved that any simple closed curve on $\S^2$ which contains the origin in its convex hull must have at least four inflection points. Another proof of Segre's theorem may be found in \cite{weiner:inflection}. For other refinements or results related to the theorem of M\"{o}bius see, \cite{Tsasaki, thorbergsson&umeharaII,pohl}.
\end{note}

\section{Proof  of Theorem \ref{thm:2}}\label{sec:thm2}
Let $M$ be a complete Riemannian surface satisfying the hypothesis of Theorem \ref{thm:2}. By Lemma \ref{lem:jackson}, and after a rescaling, we may again assume that $M$ has constant curvature $K=1$, $0$, or $-1$ which result in  the following three cases respectively:
\subsection{The elliptic case}
Recall that here $M$ is  the real projective plane $\RP^2=\S^2/\{\pm 1\}$. In this case we may construct an explicit example of a curve with only two vertices, which bounds an immersed compact surface, as shown in Figure \ref{fig:loop}.
\begin{figure}[h] 
   \centering
   \includegraphics[width=1.5in]{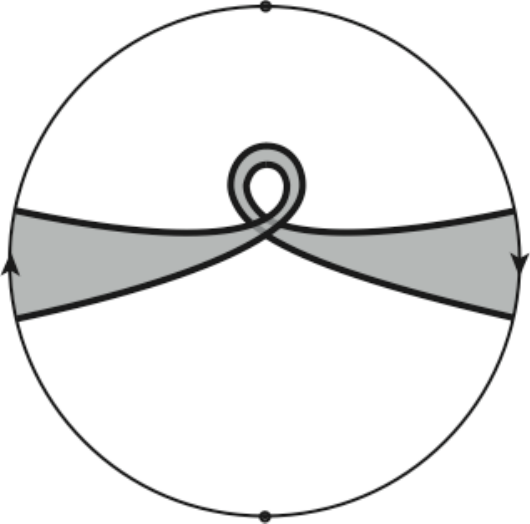} 
   \caption{}\label{fig:loop}
\end{figure}
In this picture,  $\RP^2$ is represented as a hemisphere of $\S^2$, say the northern hemisphere, with the antipodal points on its boundary, or the equator, identified, and we are looking at this hemisphere from ``above" (i.e., in a direction orthogonal to the plane of the equator). Alternatively, the above picture may be regarded as that of a unit disk, and then we may transfer the depicted region to a hemisphere via a stereographic projection, which preserves the number of vertices by Lemma \ref{lem:maeda}.
Note that the points where the curve intersects the boundary of the hemisphere are inflection points.

\subsection{The parabolic case}
In this case $M=\R^2/G$ where $G$ has at most two generators one of which must be a translation or a glide reflection \cite{stillwell}. But the composition of two glide reflections is a translation. Thus $G$ must contain a subgroup $H$ generated by a translation. Then the projection $\R^2/H\to \R^2/G$ yields a covering of $M$ by a cylinder. Thus it is enough to show that every flat cylinder contains a closed curve with only two vertices which bounds a compact immersed surface. Then the covering map yields the corresponding examples in all other topological types of flat surfaces (the twisted cylinder, the torus, and the Klein bottle). 

\begin{figure}[h] 
   \centering
   \includegraphics[width=2.75in]{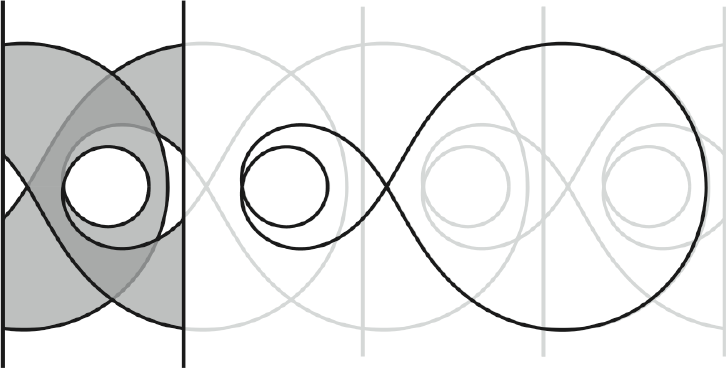} 
   \caption{}
   \label{fig:cylinder}
\end{figure}

So we may suppose that $M$ is a flat cylinder, i.e., $M=\R^2/L\Z$ which is the quotient of $\R^2$ modulo the horizontal translations $(x,y)\mapsto (x+zL,y)$. Alternatively, we may think of $M$ as the fundamental region $[0,L]\times\R\subset\R^2$ with its right and left hand sides identified. Then the curve we seek is depicted in the left hand side of Figure \ref{fig:cylinder}, and the picture on the right shows a lifting of that curve in $\R^2$.
Note that the curve on the cylinder bounds a compact immersed surface. Thus we only need to check that the given curve has only two vertices. Indeed the curve on the right is given by
$$
\gamma(t):=\frac{1}{a^2+2a\cos\left(\frac{t}{5}\right)\cos(t)+\cos\left(\frac{t}{5}\right)^2}\left(a+\cos\left(\frac{t}{5}\right)\cos(t),\;\cos\left(\frac{t}{5}\right)\sin(t)\right),
$$
where $a=9/100$. To count  the vertices of $\gamma$ note that if we invert $\gamma$ with respect to the unit circle and then translate it to the left by a distance of $a$ we obtain the curve given by $r(\theta)=\cos(\theta/5)$ in polar coordinates, see Figure \ref{fig:gencardioid}.
\begin{figure}[h] 
   \centering
   \includegraphics[width=1.1in]{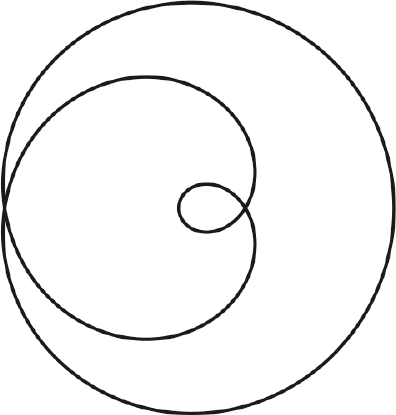} 
   \caption{}
   \label{fig:gencardioid}
\end{figure}
A straight forward  computation then shows that the derivative of the curvature of $r$ 
is given by
$$
\kappa'(\theta):=\frac{24\left(8+6\cos\left(\frac{2\theta}{5}\right)\right)\sin\left(\frac{2\theta}{5}\right)}{\left(13+12\cos\left(\frac{2\theta}{5}\right)\right)^\frac{5}{2}}.
$$
Thus $r$ 
has only two vertices which occur at $\theta=0$ and $\theta=5\pi/2$, and one may easily check that both are nondegenerate critical points of curvature. Hence $\gamma$ has only two vertices as well, since M\"{o}bius transformations preserve vertices \cite{pinkall,heil}, as we had pointed out earlier. Finally, if $\gamma$ has only two vertices as a curve in $\R^2$ then it has only  two vertices in $\H^2$ as well, by Lemma \ref{lem:maeda}. 

\begin{note}
Another example of a closed curve on the torus, or the Klein bottle, which has only two vertices but bounds a compact immersed surface is depicted in Figure \ref{fig:torus}.
\begin{figure}[h] 
   \centering
   \includegraphics[width=3.75in]{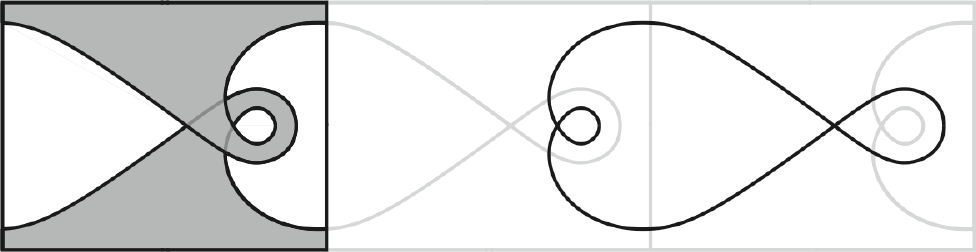} 
   \caption{}
   \label{fig:torus}
\end{figure}
Note that although this example is more simple than that of Figure \ref{fig:cylinder}, it does not work on the cylinder  or twisted cylinder because the region that it bounds would not be compact in those surfaces.
\end{note}

\begin{note}
Although the example of Figure \ref{fig:cylinder} bounds a compact immersed surface in the cylinder, the lifting of that curve does not bound any such surface in the plane. The conditions for a closed planar curve to bound a compact immersed disk were first described in the thesis of Blank \cite{blank}. See \cite{marx, francis:sphere, bailey} for further refinements and generalizations of that result.
\end{note}

\subsection{The hyperbolic case}
Recall that in this case, by Proposition \ref{prop:simpleclosed}, $M$ either has a cusp, or a simple closed geodesic. 

\subsubsection{}
Suppose first that $M$ has a cusp, then we may construct an example similar to that of Figure \ref{fig:cylinder} on $M$ because cusps are asymptotic to cylinders. More precisely, recall that a cusp is isometric to an end of a pseudosphere, which we may think of as a surface of revolution $\Sigma$ about the $z$-axis in $\R^3$.  Note that $\Sigma$ contains a meridian $m_\epsilon$ of length $\epsilon$ for every sufficiently small $\epsilon$. After a vertical translation we may assume that  $m_\epsilon$ lies on the $xy$-plane. Further, after a homothety of $\R^3$, we may assume that $m_\epsilon$ has length $2\pi$. Thus we obtain a sequence of surfaces $\Sigma_\epsilon$ which converge to  the cylinder $C$ given by $x^2+y^2=1$, within any given compact ball $B$ centered at the origin of $\R^3$, as $\epsilon\to 0$. In particular, for small $\epsilon$, $C$ and $\Sigma_\epsilon$ will be $\C^\infty$-close in $B$, and thus any curve in $B\cap C$ may be projected into $\Sigma_\epsilon$ by moving it along the normals to $C$. The new curve will be $\C^\infty$ close to the original one, and thus will have the same number of vertices as the original curve if  all the vertices of the original curve are nondegenerate critical points of curvature, which, as we verified earlier, is the case in the example of Figure \ref{fig:cylinder}.

\subsubsection{}
Now suppose that $M$ has a simple closed geodesic $\Gamma$. Then $\Gamma$ lifts to a geodesic in the upper half-plane $\H^2$, and after an isometry, we may suppose that this lifting traces the positive half of the $y$-axis. Also recall that homotheties of $\R^2$ are isometries of $\H^2$. We can use these homotheties by constructing an example similar to that of Figure \ref{fig:cylinder} along the $y$-axis, see Figure \ref{fig:cylinder2}. 
\begin{figure}[h] 
   \centering
   \includegraphics[width=2.75in]{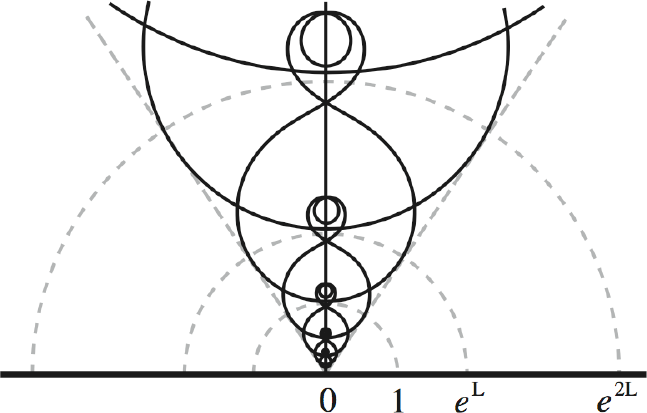} 
   \caption{}
   \label{fig:cylinder2}
\end{figure}
The curves depicted in this picture are rescalings of each other by a (Euclidean) factor of $e^L$ where $L$ is the length of $\Gamma$. This completes the proof once we note that example of Figure \ref{fig:cylinder} may be constructed for any given $L$.

\section{Proof of Theorem \ref{thm:3}}
Let $M$ be an orientable space form of genus zero, flat torus, or a rescaling of $\RP^2$, and $\Gamma\subset M$ be a simple closed curve which bounds an embedded surface. If $M=\RP^2$, or a rescaling of $\RP^2$, then $\Gamma$ has at least three inflection points by M\"{o}bius's theorem, as we pointed out in Section \ref{subsec:elliptic}. If $M$ is an orientable surface  of genus zero or a torus, then $\Gamma$ must bound a disk in $M$. Consequently $\Gamma$ lifts to a simple closed curve in the universal cover of $M$, and then it follows from Kneser's theorem and Lemma \ref{lem:maeda} that $\Gamma$ must have at least four vertices.

To complete the proof of Theorem \ref{thm:3} it remains then to show that if $M$ is not an orientable space form of genus zero, flat torus, or a rescaling of $\RP^2$, then it must contain a simple close curve $\Gamma$ which bounds a compact surface but has no more than two vertices. Again, by Lemma \ref{lem:jackson}, we may assume that $M$ has constant curvature. Further, since we already know that $M$ may not be $\S^2$ or $\RP^2$, we may suppose that the curvature of $M$ is nonpositive. The next result shows that we may assume that $M$ is orientable as well:

\begin{lem}
Any nonorientable complete surface of nonpositive curvature contains a simple closed curve with only two vertices which bounds a compact embedded surface.
\end{lem}
\begin{proof}
First recall that, by Lemma \ref{lem:mobius}, any nonorientable complete Riemannian surface $M$ contains a simple closed geodesic $\gamma\colon \R/L\to M$ which has a tubular neighborhood homeomorphic to a M\"{o}bius strip. Let $\ol\gamma$ be a lifting of $\gamma$ to the universal cover of $M$, which is $\R^2$ or $\H^2$. 

If the universal cover of $M$ is $\R^2$, i.e., $M$ has zero curvature, then we may suppose, for convenience, that $\ol\gamma$ traces the $x$-axis. Now let 
$$
\tilde\gamma_\pm(t):=\left(t,\lambda\left( \sin\left( \frac{\pi t}{L}\right)\pm \epsilon\right)\right)
$$
and note that these curves are invariant under the group of ``glide reflections" $(x,y)\mapsto (x+L,-y)$, see Figure \ref{fig:horocycle2}.
\begin{figure}[h] 
   \centering
   \includegraphics[width=3.25in]{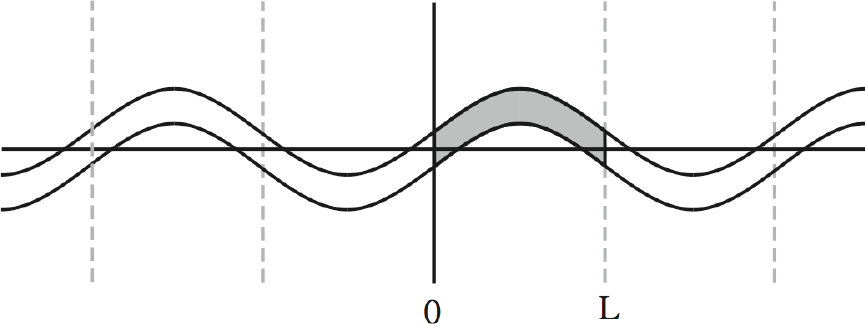} 
   \caption{}
   \label{fig:horocycle2}
\end{figure}
Thus the covering map $\pi\colon\R^2\to M$ sends these curves to a single closed curve with only two vertices in $M$ which converges to $\gamma$ with respect to the $\C^1$-norm as $\lambda\to 0$. In particular, this curve will be embedded for small $\lambda$.

If the universal cover of $M$ is $\H^2$, i.e., $M$ has negative curvature, then we may suppose that $\ol\gamma$ traces  the positive half of $y$-axis in the positive direction and $\ol\gamma(0)=(0,1)$. Now let 
$$
\tilde\gamma_\pm(t):=\left(\lambda\, t\left(\sin\left(\frac{\pi \ln(t)}{L}\right)\pm \epsilon\right),t\right)
$$ 
and again note that these curves are invariant under the group of ``glide reflections" $(x,y)\mapsto (-e^L x,e^L y)$, see Figure \ref{fig:esine2}.
\begin{figure}[h] 
   \centering
   \includegraphics[width=2.75in]{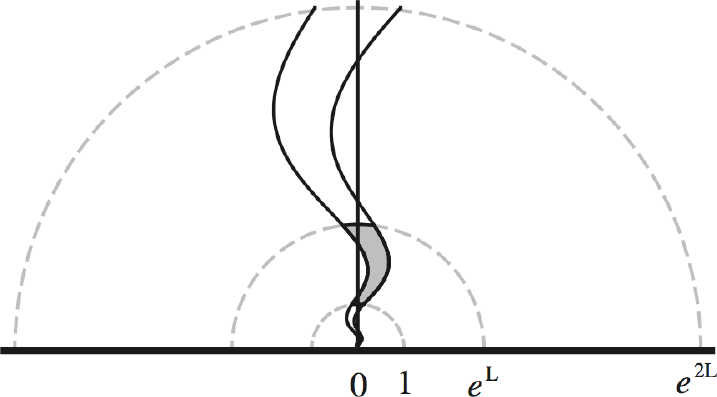} 
   \caption{}
   \label{fig:esine2}
\end{figure}
Thus the covering map $\pi\colon\H^2\to M$ once again sends these curves to the desired curve for small $\lambda$.
\end{proof}

Now recall that the only orientable parabolic $2$-dimensional space forms other than the tori are the cylinders and the Euclidean plane, both of which have genus zero. Thus we may assume that $M$ has negative curvature, and observe that 

\begin{lem}
Every complete orientable hyperbolic surface $M$ of nonzero genus contains a simple closed geodesic or a horocycle which bounds a compact subset of $M$.
\end{lem}
\begin{proof}
Since $M$ has nonzero genus, and is not a torus, it is the connected sum of a torus $T$ with a noncontractible surface $M'$; this follows from the normalization theorem of Richards \cite{richards} which states that any orientable surface is homeomorphic to the connected sum of $\S^2$ with a countable number of tori and minus a totally disconnected subset. Thus there exists a simple closed curve $\Gamma$ which divides $M$ into two components: one homeomorphic to $T$ minus a disk and the other homeomorphic to $M'$ minus a disk. Now let $\Gamma_t$ be the curve shortening flow devised in \cite{hass&scott}. Then as we argued in the proof of Proposition \ref{prop:simpleclosed}, $\Gamma_t$ either converges to a simple closed geodesic or else eventually enters every representative of an end of $M$. The latter case may happen only if that end is a cusp and $\Gamma_t$ is isotopic to a horocycle of that cusp. The proof is then complete once we recall that, by the isotopy extension lemma \cite{hirsch:book}, the topology of $M-\Gamma_t$ is independent of $t$. In particular, the closure of one of the components of $M-\Gamma_t$ will be homeomorphic to $T$ minus an open disk, which is a compact region of $M$.
\end{proof}

By Lemma \ref{lem:approx}, we may now perturb the curve given by the last result to obtain a simple closed curve with only two vertices, which completes the proof.

\section*{Appendix: More on Perturbations of Geodesics}

In Section \ref{sec:hyperbolic} we showed that any closed geodesic in a hyperbolic surface may be perturbed, in the $\C^\infty$ sense, to a closed curve with only two vertices. Here we include a more general result for orientable surfaces of constant curvature, which may not be  complete.

\begin{thm}\label{thm:perturb}
Let $\Gamma$ be a closed geodesic of length $L$ in a Riemannian surface of constant curvature $K$, which is orientable in a neighborhood of $\Gamma$. Then, every neighborhood of $\Gamma$ contains a closed curve which has only two vertices, and may be required to be arbitrarily $\C^\infty$-close to $\Gamma$, if, and only if, $K\neq (2\pi/L)^2$.
\end{thm}

The proof of this result follows from the following three lemmas. The basic idea here is again to perturb the geodesic in the direction of its normals according to a sine curve. To this end we first show that a neighborhood of $\Gamma$ in $M$ may be represented as a surface of revolution in $\R^3$. This observation, via the above theorem, shows that the only closed geodesics which cannot be perturbed to curves with only two vertices correspond to great circles in spheres (there are still many other examples of closed geodesics in noncomplete surfaces of constant positive curvature, which may be perturbed to curves with only two vertices, provided only that the length condition in the above theorem is satisfied.). 

\begin{lem}\label{lem:isometry}
Let $M$ and $M'$ be Riemannian surfaces of constant curvature  $K$ which contain simple closed geodesics $\Gamma$ and $\Gamma'$ with orientable neighborhoods. Then there exist  open neighborhoods $U$ and $U'$ of $\Gamma$ and $\Gamma'$, respectively,  and an isometry $f\colon U'\to U$ which maps $\Gamma$ to $\Gamma'$. 
\end{lem}
\begin{proof}
Fix an orientation for $M$, and for each point $p$ of $\Gamma$ let $\{e_1(p), e_2(p)\}$ be a continuous choice of basis for $T_pM$ such that $e_1$ is tangent to $\Gamma$ and $(e_1(p),e_2(p))$ is in a fixed orientation class of $T_p M$. Similarly, let $\{e'_1(p'),e'_2(p')\}$ be a continuous choice of basis along $\Gamma'$ such that $e'_1(p')$ is tangent to $\Gamma'$ and   $(e'_1(p'),e'_2(p'))$ is in a fixed orientation class of $M'$ for all $p'\in\Gamma'$. Since $M$ and $M'$ have the same constant curvature, then, as is well known, they are locally isometric. Indeed, it follows from a theorem of Cartan, that for every $p\in\Gamma$, and $p'\in\Gamma'$, there is an open neighborhood $U_p$ of $p$ and an isometry $f:=f_p\colon U_p\to M'$ such that $f(p)=p'$ and  $df(e_i(p))=e_i'(p)$, $i=1$, $2$, see \cite[p. 158]{docarmo}.

We may assume, without loss of generality, that the neighborhoods $U_p$ mentioned above are \emph{tubular}, which we define as follows. Let $N_\epsilon(\Gamma)$ be the set of all point of $M$ whose distance from $\Gamma$ is less than $\epsilon$, then by the tubular neighborhood theorem \cite{spivak:v1}, for sufficiently small $\epsilon$, $N_\epsilon(\Gamma)$ is fibrated by geodesic segments which meet $\Gamma$ orthogonally at one point. This gives rise to a natural projection map $\pi\colon N_\epsilon(\Gamma)\to\Gamma$. We say that a neighborhood $U$ of $p$ is tubular provided that there exists a connected open neighborhood $I$ of $p$ in $\Gamma$ such that $U=\pi^{-1}(I)$. 

Now let $p(t)$ and $p'(t)$ be a pair of  parametrizations of $\Gamma$ and $\Gamma'$ by arclength, and $U_t:=U_{p(t)}$ be  a tubular neighborhood of $p(t)$ such that there exists an isometry $f:=f_t\colon U_t\to M'$ with $f(p(t))=p'(t)$ and  $df(e'_i(p))=e_i'(p')$. By compactness of $\Gamma$, there will be a finite number of points $p_j=p(t_j)$ such that the  tubular neighborhoods $U_j$ cover $\Gamma$. Note that if $f_j\colon U_j\to M'$ are the corresponding local isometries, then $U'_j:=f_j(U_j)$ also form an open covering of $\Gamma'$.
Let $U:=\cup U_j$, $U':=\cup U'_j$, and define $f\colon U\to U'$ by setting $f|_{U_j}:=f_j$. It is  simple to verify that $f$ is well defined, i.e., whenever $U_j\cap U_k\neq\emptyset$, then $f_j=f_k$ on $U_j\cap U_k$, using the uniqueness of geodesics and the fact that the intersection of two tubular neighborhood is tubular, which shows that $f$ is the desired isometry.
\end{proof}

Suppose we have a surface of revolution in $\R^3$ parametrized by
\begin{equation}\label{eq:X}
X(t,\theta)=\Big(r(t)\cos(\theta),r(t)\sin(\theta), h(t)\Big).
\end{equation}
By a  \emph{neck} of $X$ we mean a meridian, i.e., a curve $\theta\mapsto X(t_0,\theta)$, which occurs at a height where the tangent to the profile curve $r$ is parallel to the axis of revolution, i.e., $r'(t_0)=0$.

\begin{lem}\label{lem:rotate}
For any constant $K$ and positive constant $L$, there exists a $\C^\infty$ surface of revolution in $\R^3$ with constant curvature $K$ and a neck of length $L$.
\end{lem}
\begin{proof}
If $K=0$, then it is obvious that we may let our surface to be a cylinder over a circle of length $L$.
If $K>0$, then, for $-\epsilon \leq t\leq\epsilon$, set
$$
r(t):=\frac{L}{2\pi}\cos\left(\frac{t}{K}\right)\quad\text{and}\quad h(t):=\int_0^t\sqrt{1-\left(\frac{L}{2\pi}\sin\left(\frac{t}{K}\right)\right)^2}\,dt.
$$
Then the surface of revolution $X$ given by \eqref{eq:X} has constant curvature $K$, see \cite[p. 483]{gray:book}. Further  $X$ has a neck at $t=0$ which has  length $L$. Similarly, if $K<0$,  set
$$
r(t):=\frac{L}{2\pi}\cosh\left(\frac{t}{K}\right)\quad\text{and}\quad h(t):=\int_0^t\sqrt{1-\left(\frac{L}{2\pi}\sinh\left(\frac{t}{K}\right)\right)^2}\,dt.
$$
Then $X$ again has constant  curvature $K$, see \cite[p. 487]{gray:book}, and  it has a neck  of length $L$ at $t=0$ . 
\end{proof}

Finally we need:

\begin{lem}\label{lem:perturb}
Let $X$ be a $\C^\infty$ surface of revolution in $\R^3$ parametrized by
\begin{equation}\label{eq:X2}
X(t,\theta):=\Big(r(t)\cos(\theta),r(t)\sin(\theta), t\Big),
\end{equation}
where $0\leq\theta\leq 2\pi$ and $-\epsilon\leq t\leq\epsilon$. Suppose that $X$ has a neck $\Gamma$ at $t=0$, and
\begin{equation}\label{eq:r}
1+r(0)r''(0)\neq 0.
\end{equation}
 Then there exists an $\epsilon>0$ such that for every $0<\lambda\leq\epsilon$ the curve
 $$
c_\lambda(\theta):=X\big(\lambda\cos(\theta), \theta\big)
 $$
 has only two vertices.
\end{lem} 
\begin{proof}
First we recall that the geodesic curvature of $c_\lambda$ is given by
$$
k_\lambda(\theta):=\left\langle \frac{T_\lambda'(\theta)}{\|c'(\theta)\|},\nu_\lambda(\theta)\right\rangle,  
$$
where $T_\lambda(\theta):=c'_\lambda(\theta)/\|c'_\lambda(\theta)\|$ is the unit tangent vector field of $c_\lambda$, and 
$\nu_\lambda(\theta)$ is a continuous normal vector field along $c_\lambda$ which is tangent to $X$.
To compute $\nu$, we may let $n(t,\theta):=\partial_t X\times \partial_\theta X/\|\partial_t X\times \partial_\theta\|$ be a unit normal vector field of $X$, and set 
$\nu_\lambda(\theta):=n(\lambda\cos(\theta),\theta)\times T_\lambda(\theta).$ A straight forward calculation then shows that 
{\small
$$
k_\lambda(\theta)=
\frac{\bar{r}' \bar{r}^2+\lambda \bar{r} \left(-\lambda  \bar{r}' \bar{r}'' \sin ^2(\theta )+\cos (\theta ) \left(\bar{r}'\right)^2+\cos (\theta )\right) +2
   \lambda ^2 \sin ^2(\theta ) \bar{r}' \left(\left(\bar{r}'\right)^2+1\right)}{\sqrt{\left(\bar{r}'\right)^2+1} \left(\bar{r}^2+\lambda ^2 \sin ^2(\theta )
   \left(\left(\bar{r}'\right)^2+1\right)\right)^{3/2}},
   $$
   }
where 
$$
\ol r:=r\big(\lambda\cos(\theta)\big),\quad
\ol r':=r'\big(\lambda\cos(\theta)\big),\quad\text{and}\quad
\ol r'':=r''\big(\lambda\cos(\theta)\big).
$$
Since by assumption $r'(0)=0$, we have $k_0(\theta)=0$ for all $\theta$. So, fixing $\theta$ and applying  Taylor's theorem, we obtain
\begin{equation}\label{eq:Taylor}
k_\lambda(\theta)=\partial_\lambda k_0(\theta)\lambda+R_\lambda(\theta)\lambda^2,
\end{equation}
where 
\begin{equation}\label{eq:k0}
\partial_\lambda k_0(\theta)=\frac{1+r(0)r''(0)}{r^2(0)}\cos(\theta),\quad\text{and}\quad R_\lambda(\theta)=\int_0^\lambda \partial_u^2 k_u(\theta)(\lambda-u)\,du.
\end{equation}
Note that \eqref{eq:Taylor} is valid for all $(\lambda,\theta)$ in $[0,1]\times[0,2\pi]$. Further, since $k_\lambda(\theta)$ depends smoothly on both $\lambda$ and $\theta$, so does $R_\lambda(\theta)$.  
Consequently, the $\C^2$-norm of $R_\lambda(\theta)$, as a function of $\theta$, has a uniform upper bound $A$ valid for all $\lambda\in[0,1]$:
$$
A:=\sup_{\lambda\in[0,1]}\|R_\lambda(\theta)\|_{\C^2}=
\sup_{\lambda\in[0,1]}\sup_{\theta\in[0,2\pi]}\{R_\lambda(\theta),R'_\lambda(\theta),R''_\lambda(\theta)\}<\infty.
$$
Furthermore, \eqref{eq:Taylor} shows that
$$
\big\|k_\lambda(\theta)-\partial_\lambda k_0(\theta)\lambda\|_{\C^2}= \|R_\lambda(\theta)\|_{\C^2}\,\lambda^2\leq A\lambda^2,
$$
where all norms are with respect to $\theta$. So 
$$
\lim_{\lambda\to 0^+}\big\|k_\lambda(\theta)-\partial_\lambda k_0(\theta)\lambda\big\|_{\C^2}=0.
$$
Combing the last equation with the computation for $\partial_\lambda k_0(\theta)$ in \eqref{eq:k0} we obtain
$$
\lim_{\lambda\to 0^+}\left\|\frac{k_\lambda(\theta)}{\lambda}-\frac{1+r(0)r''(0)}{r^2(0)}\cos(\theta)\right\|_{\C^2}=0.
$$

Now recall that $1+r(0)r''(0)\neq 0$ by assumption. Thus, for small $\lambda$,  $k_\lambda(\theta)/\lambda$ is $\C^2$-close to $C\cos(\theta)$ for some nonzero constant $C$, which yields that $k'_\lambda(\theta)/\lambda$ is $\C^1$-close to $-C\sin(\theta)$. It follows then that $k'_\lambda(\theta)/\lambda$, and consequently $k'_\lambda(\theta)$, has precisely two zeros, because $-C\sin(\theta)$ has only two zeros and at those points its derivative does not vanish. We conclude then that $k_\lambda(\theta)$ has only two local extrema, as claimed.
\end{proof}

The last three lemmas yield:

\begin{proof}[Proof of Theorem \ref{thm:perturb}]
Suppose that $K=(2\pi/L)^2$. Then, by Lemma  \ref{lem:isometry}, we may isometrically identify a neighborhood of $\Gamma$ in $M$ with a neighborhood of a great circle $\Gamma'$ in a sphere of radius $L/(2\pi)$. But any $\C^1$ perturbation of $\Gamma'$ will still be a simple closed curve, which by the classical four vertex theorem of Kneser on the sphere \cite{kneser} must have four vertices (the spherical version of the classical four vertex theorem follows from the fact that the stereographic projection preserves vertices, as had already been observed by Kneser). This proves the ``only if" part of the theorem.

Now suppose that  $K\neq(2\pi/L)^2$. 
By Lemmas \ref{lem:isometry} and  \ref{lem:rotate} we may identify a neighborhood of $\Gamma$ with a surface of revolution in $\R^3$, which we may parametrize by $X$ given in \eqref{eq:X2}, so that $\Gamma$ is identified with the neck $\Gamma'$ of $X$ at height $t=0$. Then Lemma \ref{lem:perturb} completes the proof once we can verify that $1+r(0)r''(0)\neq 0$.
To see this note that a neck of a surface of revolution is a line of curvature, i.e., it is tangent to a principal direction field of the surface, because the Gauss map $n$ sends any neck to a great circle of $\S^2$, whose plane is parallel to that of the neck, via a homothety of the neck. More precisely, if a neck has radius $r(0)$, then for any tangent vector $v$ to that neck $dn(v)=v/r(0)$, where $dn$ is the shape operator of $X$. So $X$ has constant principal curvatures $1/r(0)$ in the direction of the tangent vectors of the neck $\Gamma'$ (and with respect to the inward normal to $X$). Further, recall that a direction orthogonal to a principal direction is again a principal direction. Thus the other principal curvatures along $\Gamma'$ are given by the curvatures of the profile curve which is  $-r''(0)$ (with respect to the direction of the inward normal). So we conclude that the Gauss curvature (which is the product of principal curvatures) of $X$ along $\Gamma'$ is $K=-r''(0)/r(0)$. Note further that $L=2\pi r(0)$. Thus, by our assumption at the beginning of this paragraph, $4\pi^2\neq KL^2=-r''(0)r(0)4\pi^2$, which yields $1\neq -r''(0)r(0)$ as desired.
\end{proof}

\section*{Acknowledgement}
The author is grateful to  Rob Kusner for several useful communications and encouragement throughout this work. Also he thanks Igor Belegradek for informative discussions about the structure of hyperbolic surfaces, and Bruce Solomon for pointing out corrections to the earlier drafts of this work.


\begin{thebibliography}{10}

\bibitem{angenent}
S.~Angenent.
\newblock Inflection points, extatic points and curve shortening.
\newblock In {\em Hamiltonian systems with three or more degrees of freedom
  ({S}'{A}gar\'o, 1995)}, volume 533 of {\em NATO Adv. Sci. Inst. Ser. C Math.
  Phys. Sci.}, pages 3--10. Kluwer Acad. Publ., Dordrecht, 1999.

\bibitem{arnold:flattening}
V.~I. Arnold.
\newblock On the number of flattening points on space curves.
\newblock In {\em Sina\u\i's {M}oscow {S}eminar on {D}ynamical {S}ystems},
  volume 171 of {\em Amer. Math. Soc. Transl. Ser. 2}, pages 11--22. Amer.
  Math. Soc., Providence, RI, 1996.

\bibitem{arnold:plane}
V.~I. Arnold.
\newblock {\em Topological invariants of plane curves and caustics}, volume~5
  of {\em University Lecture Series}.
\newblock American Mathematical Society, Providence, RI, 1994.
\newblock Dean Jacqueline B. Lewis Memorial Lectures presented at Rutgers
  University, New Brunswick, New Jersey.
  
   \bibitem{bailey}
K.~D. Bailey.
\newblock Extending closed plane curves to immersions of the disk with {$n$}\
  handles.
\newblock {\em Trans. Amer. Math. Soc.}, 206:1--24, 1975.

\bibitem{blank}
S.~J. Blank.
\newblock {\em Extending immersions of the circle}.
\newblock Ph. D. Thesis, Brandeis Univ., Waltham, Mass., 1967.


\bibitem{borthwick}
D.~Borthwick.
\newblock {\em Spectral theory of infinite-area hyperbolic surfaces}, volume
  256 of {\em Progress in Mathematics}.
\newblock Birkh\"auser Boston Inc., Boston, MA, 2007.

\bibitem{cairns}
G.~Cairns, M.~{\"O}zdemir, and E.-H. Tjaden.
\newblock A counterexample to a conjecture of {U}. {P}inkall.
\newblock {\em Topology}, 31(3):557--558, 1992.

\bibitem{chou&zhu}
K.-S. Chou and X.-P. Zhu.
\newblock {\em The curve shortening problem}.
\newblock Chapman \& Hall/CRC, Boca Raton, FL, 2001.

\bibitem{costa&firer}
S.~I.~R. Costa and M.~Firer.
\newblock Four-or-more-vertex theorems for constant curvature manifolds.
\newblock In {\em Real and complex singularities (S\~ao Carlos, 1998)}, volume
  412 of {\em Chapman \& Hall/CRC Res. Notes Math.}, pages 164--172. Chapman \&
  Hall/CRC, Boca Raton, FL, 2000.

\bibitem{dahlberg}
B.~E.~J. Dahlberg.
\newblock The converse of the four vertex theorem.
\newblock {\em Proc. Amer. Math. Soc.}, 133(7):2131--2135 (electronic), 2005.

\bibitem{docarmo}
M.~P. do~Carmo.
\newblock {\em Riemannian geometry}.
\newblock Mathematics: Theory \& Applications. Birkh\"auser Boston Inc.,
  Boston, MA, 1992.
\newblock Translated from the second Portuguese edition by Francis Flaherty.


\bibitem{francis:sphere}
G.~K. Francis.
\newblock Spherical curves that bound immersed discs.
\newblock {\em Proc. Amer. Math. Soc.}, 41:87--93, 1973.

\bibitem{gluck:notices}
D.~DeTurck, H.~Gluck, D.~Pomerleano, and D.~S. Vick.
\newblock The four vertex theorem and its converse.
\newblock {\em Notices Amer. Math. Soc.}, 54(2):192--207, 2007.


\bibitem{ghomi:verticesA}
M.~Ghomi.
\newblock A Riemannian four vertex theorem for surfaces with boundary.
\newblock {\em To appear in Proc. Amer. Math. Soc}.

\bibitem{gluck:vertex}
H.~Gluck.
\newblock The converse to the four vertex theorem.
\newblock {\em Enseignement Math. (2)}, 17:295--309, 1971.

\bibitem{gray:book}
A.~Gray.
\newblock {\em Modern differential geometry of curves and surfaces with
  {M}athematica}.
\newblock CRC Press, Boca Raton, FL, second edition, 1998.


\bibitem{grayson}
M.~A. Grayson.
\newblock Shortening embedded curves.
\newblock {\em Ann. of Math. (2)}, 129(1):71--111, 1989.

\bibitem{guggenheimer}
H.~W. Guggenheimer.
\newblock {\em Differential geometry}.
\newblock Dover Publications Inc., New York, 1977.
\newblock Corrected reprint of the 1963 edition, Dover Books on Advanced
  Mathematics.

\bibitem{hass&scott}
J.~Hass and P.~Scott.
\newblock Shortening curves on surfaces.
\newblock {\em Topology}, 33(1):25--43, 1994.

\bibitem{heil}
E.~Heil.
\newblock Some vertex theorems proved by means of {M}oebius transformations.
\newblock {\em Ann. Mat. Pura Appl. (4)}, 85:301--306, 1970.

\bibitem{heil2}
E.~Heil.
\newblock A four-vertex theorem for space curves.
\newblock {\em Math. Pannon.}, 10(1):123--132, 1999.

\bibitem{hirsch:book}
M.~W. Hirsch.
\newblock {\em Differential topology}.
\newblock Springer-Verlag, New York, 1994.
\newblock Corrected reprint of the 1976 original.

\bibitem{jackson:bulletin}
S.~B. Jackson.
\newblock Vertices for plane curves.
\newblock {\em Bull. Amer. Math. Soc.}, 50:564--478, 1944.

\bibitem{jackson:vertices}
S.~B. Jackson.
\newblock The four-vertex theorem for surfaces of constant curvature.
\newblock {\em Amer. J. Math.}, 67:563--582, 1945.

\bibitem{kneser}
A.~Kneser.
\newblock Bemerkungen \"{u}ber die anzahl der extrema des kr\"{u}mmung auf
  geschlossenen kurven und \"{u}ber verwandte fragen in einer night
  eucklidischen geometrie.
\newblock In {\em Festschrift Heinrich Weber}, pages 170--180. Teubner, 1912.

\bibitem{maeda}
M.~Maeda.
\newblock The four-or-more vertex theorems in {$2$}-dimensional space forms.
\newblock {\em Nat. Sci. J. Fac. Educ. Hum. Sci. Yokohama Natl. Univ.},
  (1):43--46, 1998.
  
 \bibitem{marx}
M.~L. Marx.
\newblock Extensions of normal immersions of {$S\sp{1}$} into {$R\sp{2}$}.
\newblock {\em Trans. Amer. Math. Soc.}, 187:309--326, 1974.


\bibitem{mobius}
A.~F. M\"{o}bius.
\newblock \"{U}ber die grundformen der linien der dritten ordnung.
\newblock In {\em Gesammelte Werke II}, pages 89--176. Verlag von S. Hirzel,
  Leipzig, 1886.

\bibitem{montiel&ros}
S.~Montiel and A.~Ros.
\newblock {\em Curves and surfaces}, volume~69 of {\em Graduate Studies in
  Mathematics}.
\newblock American Mathematical Society, Providence, RI, 2005.
\newblock Translated and updated from the 1998 Spanish edition by the authors.

\bibitem{mukhopadhyaya}
S.~Mukhopadhyaya.
\newblock New methods in the geometry of a plane arc.
\newblock {\em Bull. Calcutta Math. Soc. I}, pages 31--37, 1909.

\bibitem{osserman:vertex}
R.~Osserman.
\newblock The four-or-more vertex theorem.
\newblock {\em Amer. Math. Monthly}, 92(5):332--337, 1985.

\bibitem{ovsienko&tabachnikov}
V.~Ovsienko and S.~Tabachnikov.
\newblock {\em Projective differential geometry old and new}, volume 165 of
  {\em Cambridge Tracts in Mathematics}.
\newblock Cambridge University Press, Cambridge, 2005.
\newblock From the Schwarzian derivative to the cohomology of diffeomorphism
  groups.

\bibitem{pak:book}
I.~Pak.
\newblock {\em Lectures on discrete and polyhedral geometry}.
\newblock www.math.umn.edu/$\sim$pak/, 2008.

\bibitem{pinkall}
U.~Pinkall.
\newblock On the four-vertex theorem.
\newblock {\em Aequationes Math.}, 34(2-3):221--230, 1987.

\bibitem{pohl}
W.~F. Pohl.
\newblock On a theorem related to the four-vertex theorem.
\newblock {\em Ann. of Math. (2)}, 84:356--367, 1966.

\bibitem{rmbc}
E. Raphael, J.-M. di Meglio, M. Berger and E. Calabi.
\newblock Convex Particles at Interfaces,
{\em J. Phys. I France}  2:571--579, 1992.


\bibitem{richards}
I.~Richards.
\newblock On the classification of noncompact surfaces.
\newblock {\em Trans. Amer. Math. Soc.}, 106:259--269, 1963.


\bibitem{ratcliffe}
J.~G. Ratcliffe.
\newblock {\em Foundations of hyperbolic manifolds}, volume 149 of {\em
  Graduate Texts in Mathematics}.
\newblock Springer, New York, second edition, 2006.


\bibitem{Ssasaki}
S.~Sasaki.
\newblock The minimum number of points of inflexion of closed curves in the
  projective plane.
\newblock {\em T\^ohoku Math. J. (2)}, 9:113--117, 1957.

\bibitem{Tsasaki}
T.~Sasaki.
\newblock Inflection points and affine vertices of closed curves on
  {$2$}-dimensional affine flat tori.
\newblock {\em Results Math.}, 27(1-2):129--140, 1995.
\newblock Festschrift dedicated to Katsumi Nomizu on his 70th birthday (Leuven,
  1994; Brussels, 1994).

\bibitem{sedykh:vertex}
V.~D. Sedykh.
\newblock Four vertices of a convex space curve.
\newblock {\em Bull. London Math. Soc.}, 26(2):177--180, 1994.

\bibitem{segre:inflection}
B.~Segre.
\newblock Alcune propriet\`a differenziali in grande delle curve chiuse
  sghembe.
\newblock {\em Rend. Mat. (6)}, 1:237--297, 1968.

\bibitem{spivak:v1}
M.~Spivak.
\newblock {\em A comprehensive introduction to differential geometry. {V}ol.
  {I}}.
\newblock Publish or Perish Inc., Wilmington, Del., second edition, 1979.


\bibitem{stillwell}
J.~Stillwell.
\newblock {\em Geometry of surfaces}.
\newblock Universitext. Springer-Verlag, New York, 1992.


\bibitem{sussman}
B.~S{\"u}ssmann.
\newblock Curve shortening and the four-vertex theorem.
\newblock {\em Port. Math. (N.S.)}, 62(3):269--288, 2005.

\bibitem{thorbergsson&umehara}
G.~Thorbergsson and M.~Umehara.
\newblock A unified approach to the four vertex theorems. {II}.
\newblock In {\em Differential and symplectic topology of knots and curves},
  volume 190 of {\em Amer. Math. Soc. Transl. Ser. 2}, pages 229--252. Amer.
  Math. Soc., Providence, RI, 1999.

\bibitem{thorbergsson&umeharaII}
G.~Thorbergsson and M.~Umehara.
\newblock Inflection points and double tangents on anti-convex curves in the
  real projective plane.
\newblock {\em Tohoku Math. J. (2)}, 60(2):149--181, 2008.

\bibitem{umehara}
M.~Umehara.
\newblock {$6$}-vertex theorem for closed planar curve which bounds an immersed
  surface with nonzero genus.
\newblock {\em Nagoya Math. J.}, 134:75--89, 1994.

\bibitem{umehara2}
M.~Umehara.
\newblock A unified approach to the four vertex theorems. {I}.
\newblock In {\em Differential and symplectic topology of knots and curves},
  volume 190 of {\em Amer. Math. Soc. Transl. Ser. 2}, pages 185--228. Amer.
  Math. Soc., Providence, RI, 1999.

\bibitem{uribe-vargas}
R.~Uribe-Vargas.
\newblock Four-vertex theorems, {S}turm theory and {L}agrangian singularities.
\newblock {\em Math. Phys. Anal. Geom.}, 7(3):223--237, 2004.

\bibitem{weiner:inflection}
J.~L. Weiner.
\newblock Global properties of spherical curves.
\newblock {\em J. Differential Geom.}, 12(3):425--434, 1977.

\end{thebibliography}

\end{document}